\documentclass{amsart}
\usepackage{amsmath}
\usepackage{amsfonts}

\newtheorem{theorem}{Theorem}
\theoremstyle{plain}

\newtheorem{corollary}{Corollary}

\newtheorem{definition}{Definition}

\newtheorem{lemma}{Lemma}

\newtheorem{proposition}{Proposition}
\newtheorem{remark}{Remark}

\numberwithin{equation}{section}
\input{tcilatex}

\begin{document}
\title[Generalization of some integral inequalities]{New estimates on
generalization of some integral inequalities for $s-$convex functions and
their applications}
\author{\.{I}mdat \.{I}\c{s}can}
\address{Department of Mathematics, Faculty of Arts and Sciences,\\
Giresun University, 28100, Giresun, Turkey.}
\email{imdat.iscan@giresun.edu.tr}
\date{July 21, 2012}

\begin{abstract}
In this paper, a new identity for differentiable functions is derived. Thus
we can obtain new estimates on generalization of Hadamard,Ostrowski and
Simpson type inequalities for functions whose derivatives in absolute value
at certain power are $s-$convex (in the second sense). Some applications to
special means of real numbers are also given.
\end{abstract}

\maketitle

\section{Introduction}

Let $f:I\subseteq \mathbb{R\rightarrow R}$ be a convex function defined on
the interval $I$ of real numbers and $a,b\in I$ with $a<b$. The following
inequality%
\begin{equation}
f\left( \frac{a+b}{2}\right) \leq \frac{1}{b-a}\dint\limits_{a}^{b}f(x)dx%
\leq \frac{f(a)+f(b)}{2}\text{.}  \label{1-1}
\end{equation}

holds. This double inequality is known in the literature as Hermite-Hadamard
integral inequality for convex functions \cite{DP00}. Note that some of the
classical inequalities for means can be derived from (\ref{1-1}) for
appropriate particular selections of the mapping $f$. Both inequalities hold
in the reversed direction if f is concave.

Let $f:I\subseteq \mathbb{R\rightarrow R}$ be a mapping differentiable in $%
I^{\circ },$ the interiorof I, and let $a,b\in I^{\circ }$ with $a<b.$ If $%
\left\vert f^{\prime }(x)\right\vert \leq M,$ $x\in \left[ a,b\right] ,$
then we the following inequality holds%
\begin{equation*}
\left\vert f(x)-\frac{1}{b-a}\dint\limits_{a}^{b}f(t)dt\right\vert \leq 
\frac{M}{b-a}\left[ \frac{\left( x-a\right) ^{2}+\left( b-x\right) ^{2}}{2}%
\right]
\end{equation*}%
for all $x\in \left[ a,b\right] .$ The constant $\frac{1}{4}$ is the best
possible in the sense that it cannot be replaced by a smaller one. This
result is known in the literature as the Ostrowski inequality \cite{DR02}.

The following inequality is well known in the literature as Simpson's
inequality .

Let $f:\left[ a,b\right] \mathbb{\rightarrow R}$ be a four times
continuously differentiable mapping on $\left( a,b\right) $ and $\left\Vert
f^{(4)}\right\Vert _{\infty }=\underset{x\in \left( a,b\right) }{\sup }%
\left\vert f^{(4)}(x)\right\vert <\infty .$ Then the following inequality
holds:%
\begin{equation*}
\left\vert \frac{1}{3}\left[ \frac{f(a)+f(b)}{2}+2f\left( \frac{a+b}{2}%
\right) \right] -\frac{1}{b-a}\dint\limits_{a}^{b}f(x)dx\right\vert \leq 
\frac{1}{2880}\left\Vert f^{(4)}\right\Vert _{\infty }\left( b-a\right) ^{2}.
\end{equation*}%
\ \qquad In recent years many authors have studied error estimations for
Simpson's inequality; for refinements, counterparts, generalizations and new
Simpson's type inequalities, see \cite{ADD09,SA11,SSO10,SSO10a} and therein.

In \cite{B78}, Breckner introduced s-convex functions as a generalization of
convex functions as follows:

\begin{definition}
Let $s\in (0,1]$ be a fixed real number. A function $f:[0,\infty
)\rightarrow \lbrack 0,\infty )$ is said to be $s-$convex (in the second
sense),or that $f$ belongs to the class $K_{s}^{2}$, if \ 
\begin{equation}
f\left( tx+(1-t)y\right) \leq t^{s}f(x)+(1-t)^{s}f(y)  \label{1-2}
\end{equation}
\end{definition}

for all $x,y\in \lbrack 0,\infty )$ and $t\in \lbrack 0,1]$.

If inequality (\ref{1-2}) is reversed, then $f$ is said to be $s-$concave
(in the second sense). Of course, s-convexity means just convexity when s =
1.

In \cite{DF99}, Dragomir and Fitzpatrick proved a variant of Hadamard's
inequality which holds for $s-$convex functions in the second sense.

\begin{theorem}
Suppose that $f:[0,\infty )\rightarrow \lbrack 0,\infty )$ is an $s-$convex
function in the second sense, where $s\in (0,1)$, and let $a,b\in \lbrack
0,\infty )$, $a<b.$If $f\in L[a,b]$ then the following inequalities hold%
\begin{equation}
2^{s-1}f\left( \frac{a+b}{2}\right) \leq \frac{1}{b-a}\dint%
\limits_{a}^{b}f(x)dx\leq \frac{f(a)+f(b)}{s+1}.  \label{1-3}
\end{equation}
\end{theorem}

Both inequalities hold in the reversed direction if f is $s-$concave. The
constant $k=\frac{1}{s+1}$ is the best possible in the second inequality in (%
\ref{1-3}).

\section{Main results}

In order to prove our main theorems, we need the following Lemma.

\begin{lemma}
\label{1.1}Let $f:I\subseteq \mathbb{R\rightarrow R}$ be a differentiable
mapping on $I^{\circ }$ such that $f^{\prime }\in L[a,b]$, where $a,b\in I$
with $a<b$ and $\theta ,\lambda \in \left[ 0,1\right] $. Then the following
equality holds:%
\begin{eqnarray*}
&&\left( 1-\theta \right) \left( \lambda f(a)+\left( 1-\lambda \right)
f(b)\right) +\theta f(\left( 1-\lambda \right) a+\lambda b)-\frac{1}{b-a}%
\dint\limits_{a}^{b}f(x)dx \\
&=&\left( b-a\right) \left[ -\lambda ^{2}\dint\limits_{0}^{1}\left( t-\theta
\right) f^{\prime }\left( ta+\left( 1-t\right) \left[ \left( 1-\lambda
\right) a+\lambda b\right] \right) dt\right. \\
&&\left. +\left( 1-\lambda \right) ^{2}\dint\limits_{0}^{1}\left( t-\theta
\right) f^{\prime }\left( tb+\left( 1-t\right) \left[ \left( 1-\lambda
\right) a+\lambda b\right] \right) dt\right] .
\end{eqnarray*}
\end{lemma}

\begin{proof}
Firstly suppose that $\lambda \in \left( 0,1\right) $ and let $C=\left(
1-\lambda \right) a+\lambda b.$ 
\begin{equation*}
I=-\lambda ^{2}\dint\limits_{0}^{1}\left( t-\theta \right) f^{\prime }\left(
ta+\left( 1-t\right) C\right) dt+\left( 1-\lambda \right)
^{2}\dint\limits_{0}^{1}\left( t-\theta \right) f^{\prime }\left( tb+\left(
1-t\right) C\right) dt
\end{equation*}%
integrating by parts, we get%
\begin{eqnarray*}
I &=&\left. \left( t-\theta \right) \frac{\lambda f\left( ta+C\left(
1-t\right) \right) }{b-a}\right\vert _{0}^{1}-\frac{\lambda }{b-a}%
\dint\limits_{0}^{1}f\left( ta+C\left( 1-t\right) \right) dt \\
&&\left. +\left( t-\theta \right) \frac{\left( 1-\lambda \right) f\left(
tb+C\left( 1-t\right) \right) }{b-a}\right\vert _{0}^{1}-\frac{\left(
1-\lambda \right) }{b-a}\dint\limits_{0}^{1}f\left( tb+C\left( 1-t\right)
\right) dt \\
&=&\lambda \left( 1-\theta \right) \frac{f\left( a\right) }{b-a}+\lambda
\theta \frac{f\left( C\right) }{b-a}+\left( 1-\lambda \right) \left(
1-\theta \right) \frac{f\left( b\right) }{b-a}+\left( 1-\lambda \right)
\theta \frac{f\left( C\right) }{b-a} \\
&&-\frac{1}{b-a}\left[ \lambda \dint\limits_{0}^{1}f\left( ta+C\left(
1-t\right) \right) dt+\left( 1-\lambda \right) \dint\limits_{0}^{1}f\left(
tb+C\left( 1-t\right) \right) dt\right]
\end{eqnarray*}%
Setting $x=ta+C\left( 1-t\right) ,$ $dx=\lambda \left( a-b\right) dt,$and $%
x=tb+C\left( 1-t\right) ,$ $dx=\mu \left( b-a\right) dt$ respectively, we
obtain%
\begin{eqnarray*}
&&\lambda \dint\limits_{0}^{1}f\left( ta+C\left( 1-t\right) \right)
dt+\left( 1-\lambda \right) \dint\limits_{0}^{1}f\left( tb+C\left(
1-t\right) \right) dt \\
&=&\frac{1}{b-a}\left[ \dint\limits_{a}^{C}f\left( x\right)
dx+\dint\limits_{C}^{b}f\left( x\right) dx\right] =\frac{1}{b-a}%
\dint\limits_{a}^{b}f\left( x\right) dx
\end{eqnarray*}%
and so we have%
\begin{equation*}
\left( b-a\right) I=\left( 1-\theta \right) \left( \lambda f(a)+\left(
1-\lambda \right) f(b)\right) +\theta f(C)-\frac{1}{b-a}\dint%
\limits_{a}^{b}f(x)dx
\end{equation*}%
which gives the desired representation (\ref{2-1}).

Secondly suppose that $\lambda \in \left\{ 0,1\right\} .$ The identities%
\begin{equation*}
\left( 1-\theta \right) f(b)+\theta f(a)-\frac{1}{b-a}\dint%
\limits_{a}^{b}f(x)dx=\left( b-a\right) \dint\limits_{0}^{1}\left( t-\theta
\right) f^{\prime }\left( tb+\left( 1-t\right) a\right) dt
\end{equation*}%
and%
\begin{equation*}
\left( 1-\theta \right) f(a)+\theta f(b)-\frac{1}{b-a}\dint%
\limits_{a}^{b}f(x)dx=\left( b-a\right) \dint\limits_{0}^{1}\left( \theta
-t\right) f^{\prime }\left( ta+\left( 1-t\right) b\right) dt.
\end{equation*}%
can be proved by performing an integration by parts in the integrals from
the right side and changing the variable.
\end{proof}

\begin{theorem}
\label{2.1}Let $f:I\subseteq \lbrack 0,\infty )\mathbb{\rightarrow R}$ be a
differentiable mapping on $I^{\circ }$ such that $f^{\prime }\in L[a,b]$,
where $a,b\in I^{\circ }$ with $a<b$ and $\theta ,\lambda \in \left[ 0,1%
\right] $. If $\left\vert f^{\prime }\right\vert ^{q}$ is $s-$convex on $%
[a,b]$, $q\geq 1,$ then the following inequality holds:%
\begin{equation*}
\left\vert \left( 1-\theta \right) \left( \lambda f(a)+\left( 1-\lambda
\right) f(b)\right) +\theta f(\left( 1-\lambda \right) a+\lambda b)-\frac{1}{%
b-a}\dint\limits_{a}^{b}f(x)dx\right\vert
\end{equation*}%
\begin{eqnarray}
&\leq &\left( b-a\right) A_{1}^{1-\frac{1}{q}}(\theta )\left\{ \lambda ^{2}%
\left[ \left\vert f^{\prime }\left( a\right) \right\vert ^{q}A_{2}(\theta
,s)+\left\vert f^{\prime }\left( C\right) \right\vert ^{q}A_{3}(\theta ,s)%
\right] ^{\frac{1}{q}}\right.  \notag \\
&&\left. +\left( 1-\lambda \right) ^{2}\left[ \left\vert f^{\prime }\left(
b\right) \right\vert ^{q}A_{2}(\theta ,s)+\left\vert f^{\prime }\left(
C\right) \right\vert ^{q}A_{3}(\theta ,s)\right] ^{\frac{1}{q}}\right\}
\label{2-1}
\end{eqnarray}%
where 
\begin{eqnarray*}
A_{1}(\theta ) &=&\theta ^{2}-\theta +\frac{1}{2} \\
A_{2}(\theta ,s) &=&\frac{2\theta ^{s+2}}{(s+1)(s+2)}-\frac{\theta }{s+1}+%
\frac{1}{s+2}, \\
A_{3}(\theta ,s) &=&\frac{2\left( 1-\theta \right) ^{s+2}}{(s+1)(s+2)}-\frac{%
1-\theta }{s+1}+\frac{1}{s+2}.
\end{eqnarray*}%
and $C=\left( 1-\lambda \right) a+\lambda b.$
\end{theorem}

\begin{proof}
Suppose that $q\geq 1$ and $C=\left( 1-\lambda \right) a+\lambda b.$ From
Lemma \ref{1.1} and using the well known power mean inequality, we have%
\begin{eqnarray*}
&&\left\vert \left( 1-\theta \right) \left( \lambda f(a)+\left( 1-\lambda
\right) f(b)\right) +\theta f(C)-\frac{1}{b-a}\dint\limits_{a}^{b}f(x)dx%
\right\vert \leq \left( b-a\right) \\
&&\left[ \lambda ^{2}\dint\limits_{0}^{1}\left\vert t-\theta \right\vert
\left\vert f^{\prime }\left( ta+\left( 1-t\right) C\right) \right\vert
dt+\left( 1-\lambda \right) ^{2}\dint\limits_{0}^{1}\left\vert t-\theta
\right\vert \left\vert f^{\prime }\left( tb+\left( 1-t\right) C\right)
\right\vert dt\right] \\
&\leq &\left( b-a\right) \left\{ \lambda ^{2}\left(
\dint\limits_{0}^{1}\left\vert t-\theta \right\vert dt\right) ^{1-\frac{1}{q}%
}\left( \dint\limits_{0}^{1}\left\vert t-\theta \right\vert \left\vert
f^{\prime }\left( ta+\left( 1-t\right) C\right) \right\vert ^{q}dt\right) ^{%
\frac{1}{q}}\right.
\end{eqnarray*}%
\begin{equation}
\left. +\left( 1-\lambda \right) ^{2}\left( \dint\limits_{0}^{1}\left\vert
t-\theta \right\vert dt\right) ^{1-\frac{1}{q}}\left(
\dint\limits_{0}^{1}\left\vert t-\theta \right\vert \left\vert f^{\prime
}\left( tb+\left( 1-t\right) C\right) \right\vert ^{q}dt\right) ^{\frac{1}{q}%
}\right\} .  \label{2-1a}
\end{equation}

Since $\left\vert f^{\prime }\right\vert ^{q}$ is $s-$convex on $[a,b],$ we
know that for $t\in \left[ 0,1\right] $%
\begin{equation}
\left\vert f^{\prime }\left( ta+C\left( 1-t\right) \right) \right\vert
^{q}\leq t^{s}\left\vert f^{\prime }\left( a\right) \right\vert ^{q}+\left(
1-t\right) ^{s}\left\vert f^{\prime }\left( C\right) \right\vert ^{q}
\label{2-11a}
\end{equation}%
and%
\begin{equation}
\left\vert f^{\prime }\left( tb+C\left( 1-t\right) \right) \right\vert
^{q}\leq t^{s}\left\vert f^{\prime }\left( b\right) \right\vert ^{q}+\left(
1-t\right) ^{s}\left\vert f^{\prime }\left( C\right) \right\vert ^{q}.
\label{2-11b}
\end{equation}%
Hence, by simple computation%
\begin{eqnarray}
&&\dint\limits_{0}^{1}\left\vert t-\theta \right\vert \left\vert f^{\prime
}\left( ta+\left( 1-t\right) C\right) \right\vert ^{q}dt  \notag \\
&\leq &\dint\limits_{0}^{1}\left\vert t-\theta \right\vert \left[
t^{s}\left\vert f^{\prime }\left( a\right) \right\vert ^{q}+\left(
1-t\right) ^{s}\left\vert f^{\prime }\left( C\right) \right\vert ^{q}\right]
dt  \notag \\
&=&\left\vert f^{\prime }\left( a\right) \right\vert
^{q}\dint\limits_{0}^{1}\left\vert t-\theta \right\vert t^{s}dt+\left\vert
f^{\prime }\left( C\right) \right\vert ^{q}\dint\limits_{0}^{1}\left\vert
t-\theta \right\vert \left( 1-t\right) ^{s}dt  \notag \\
&=&\left\vert f^{\prime }\left( a\right) \right\vert ^{q}A_{2}(\theta
,s)+\left\vert f^{\prime }\left( C\right) \right\vert ^{q}A_{3}(\theta ,s)
\label{2-1b}
\end{eqnarray}%
\begin{eqnarray}
&&\dint\limits_{0}^{1}\left\vert t-\theta \right\vert \left\vert f^{\prime
}\left( tb+\left( 1-t\right) C\right) \right\vert ^{q}dt  \notag \\
&\leq &\dint\limits_{0}^{1}\left\vert t-\theta \right\vert \left[
t^{s}\left\vert f^{\prime }\left( b\right) \right\vert ^{q}+\left(
1-t\right) ^{s}\left\vert f^{\prime }\left( C\right) \right\vert ^{q}\right]
dt  \notag \\
&=&\left\vert f^{\prime }\left( b\right) \right\vert
^{q}\dint\limits_{0}^{1}\left\vert t-\theta \right\vert t^{s}dt+\left\vert
f^{\prime }\left( C\right) \right\vert ^{q}\dint\limits_{0}^{1}\left\vert
t-\theta \right\vert \left( 1-t\right) ^{s}dt  \notag \\
&=&\left\vert f^{\prime }\left( b\right) \right\vert ^{q}A_{2}(\theta
,s)+\left\vert f^{\prime }\left( C\right) \right\vert ^{q}A_{3}(\theta ,s)
\label{2-1c}
\end{eqnarray}%
and%
\begin{equation}
\dint\limits_{0}^{1}\left\vert t-\theta \right\vert dt=\theta ^{2}-\theta +%
\frac{1}{2}.  \label{2-1d}
\end{equation}%
Thus, using (\ref{2-1b}),(\ref{2-1c}) and (\ref{2-1d}) in (\ref{2-1a}), we
obtain the inequality (\ref{2-1}). This completes the proof.
\end{proof}

\begin{corollary}
Under the assumptions of Theorem \ref{2.1} with $q=1,$ we have 
\begin{eqnarray*}
&&\left\vert \left( 1-\theta \right) \left( \lambda f(a)+\left( 1-\lambda
\right) f(b)\right) +\theta f(\left( 1-\lambda \right) a+\lambda b)-\frac{1}{%
b-a}\dint\limits_{a}^{b}f(x)dx\right\vert \\
&\leq &\left( b-a\right) \left\{ A_{2}(\theta ,s)\left( \lambda
^{2}\left\vert f^{\prime }\left( a\right) \right\vert +\left( 1-\lambda
\right) ^{2}\left\vert f^{\prime }\left( b\right) \right\vert \right)
+A_{3}(\theta ,s)\left( 2\lambda ^{2}-2\lambda +1\right) \left\vert
f^{\prime }\left( C\right) \right\vert \right\}
\end{eqnarray*}%
where $A_{2}(\theta ,s)$ and $A_{3}(\theta ,s)$ are defined as in Theorem %
\ref{2.1}.
\end{corollary}

\begin{corollary}
Under the assumptions of Theorem \ref{2.1} with $s=1$ we have%
\begin{equation*}
\left\vert \left( 1-\theta \right) \left( \lambda f(a)+\left( 1-\lambda
\right) f(b)\right) +\theta f(C)-\frac{1}{b-a}\dint\limits_{a}^{b}f(x)dx%
\right\vert
\end{equation*}%
\begin{eqnarray*}
&\leq &\left( b-a\right) A_{1}^{1-\frac{1}{q}}(\theta )\left\{ \lambda ^{2}%
\left[ \left\vert f^{\prime }\left( a\right) \right\vert ^{q}A_{2}(\theta
,1)+\left\vert f^{\prime }\left( C\right) \right\vert ^{q}A_{3}(\theta ,1)%
\right] ^{\frac{1}{q}}\right. \\
&&\left. +\left( 1-\lambda \right) ^{2}\left[ \left\vert f^{\prime }\left(
b\right) \right\vert ^{q}A_{2}(\theta ,1)+\left\vert f^{\prime }\left(
C\right) \right\vert ^{q}A_{3}(\theta ,1)\right] ^{\frac{1}{q}}\right\}
\end{eqnarray*}%
where 
\begin{eqnarray*}
A_{1}(\theta ) &=&\theta ^{2}-\theta +\frac{1}{2} \\
A_{2}(\theta ,1) &=&\frac{1+\theta ^{3}}{3}-\frac{\theta }{2}, \\
A_{3}(\theta ,1) &=&\frac{1+\left( 1-\theta \right) ^{3}}{3}-\frac{1-\theta 
}{2},
\end{eqnarray*}%
and $C=\left( 1-\lambda \right) a+\lambda b.$
\end{corollary}

\begin{corollary}
Under the assumptions of Theorem \ref{2.1} with $\theta =1,$ then we have
following generalized midpoint type inequality%
\begin{eqnarray*}
&&\left\vert f(\left( 1-\lambda \right) a+\lambda b)-\frac{1}{b-a}%
\dint\limits_{a}^{b}f(x)dx\right\vert \\
&\leq &\frac{b-a}{2}\left( \frac{2}{(s+1)(s+2)}\right) ^{\frac{1}{q}}\left\{
\lambda ^{2}\left( \left\vert f^{\prime }(a)\right\vert ^{q}+\left(
s+1\right) \left\vert f^{\prime }(C)\right\vert ^{q}\right) ^{\frac{1}{q}%
}\right. \\
&&\left. +\left( 1-\lambda \right) ^{2}\left( \left\vert f^{\prime
}(b)\right\vert ^{q}+\left( s+1\right) \left\vert f^{\prime }(C)\right\vert
^{q}\right) ^{\frac{1}{q}}\right\} ,
\end{eqnarray*}%
where $C=\left( 1-\lambda \right) a+\lambda b.$
\end{corollary}

\begin{corollary}
Under the assumptions of Theorem \ref{2.1} with $\theta =1,$ if $\left\vert
f^{\prime }(x)\right\vert \leq M,$ $x\in \left[ a,b\right] ,$ then we have
the following Ostrowski type inequality%
\begin{equation}
\left\vert f(x)-\frac{1}{b-a}\dint\limits_{a}^{b}f(u)du\right\vert \leq
M\left( \frac{2}{s+1}\right) ^{\frac{1}{q}}\left[ \frac{\left( x-a\right)
^{2}+\left( b-x\right) ^{2}}{2\left( b-a\right) }\right]  \label{2-1e}
\end{equation}%
for each $x\in \left[ a,b\right] .$
\end{corollary}

\begin{proof}
For each $x\in \left[ a,b\right] $, there exist $\lambda _{x}\in \left[ 0,1%
\right] $ such that $x=\left( 1-\lambda _{x}\right) a+\lambda _{x}b.$ Hence
we have $\lambda _{x}=\frac{x-a}{b-a}$ and $1-\lambda _{x}=\frac{b-x}{b-a}.$
Therefore for each $x\in \left[ a,b\right] ,$ from the inequality (\ref{2-1}%
) we obtain the inequality (\ref{2-1e}).
\end{proof}

\begin{remark}
We note that the inequality (\ref{2-1e}) is the same of the inequality in 
\cite[Theorem 4]{ADDC10}.
\end{remark}

\begin{corollary}
Under the assumptions of Theorem \ref{2.1} with $\theta =0,$ then we have
following generalized trapezoid type inequality%
\begin{eqnarray*}
&&\left\vert \lambda f(a)+\left( 1-\lambda \right) f(b)-\frac{1}{b-a}%
\dint\limits_{a}^{b}f(x)dx\right\vert \\
&\leq &\frac{b-a}{2}\left( \frac{2}{(s+1)(s+2)}\right) ^{\frac{1}{q}}\left\{
\lambda ^{2}\left( \left( s+1\right) \left\vert f^{\prime }(a)\right\vert
^{q}+\left\vert f^{\prime }(C)\right\vert ^{q}\right) ^{\frac{1}{q}}\right.
\\
&&\left. +\left( 1-\lambda \right) ^{2}\left( \left( s+1\right) \left\vert
f^{\prime }(b)\right\vert ^{q}+\left\vert f^{\prime }(C)\right\vert
^{q}\right) ^{\frac{1}{q}}\right\} ,
\end{eqnarray*}%
where $C=\left( 1-\lambda \right) a+\lambda b.$
\end{corollary}

\begin{corollary}
Under the assumptions of Theorem \ref{2.1} with $\lambda =\frac{1}{2}$ and $%
\theta =\frac{2}{3},$ then we have the following Simpson type inequality%
\begin{eqnarray*}
&&\left\vert \frac{1}{6}\left[ f(a)+4f\left( \frac{a+b}{2}\right) +f(b)%
\right] -\frac{1}{b-a}\dint\limits_{a}^{b}f(x)dx\right\vert \\
&\leq &\frac{b-a}{4}\left( \frac{5}{18}\right) ^{1-\frac{1}{q}}\left\{
\left( A_{2}(\frac{2}{3},s)\left\vert f^{\prime }(a)\right\vert ^{q}+A_{3}(%
\frac{2}{3},s)\left\vert f^{\prime }(\frac{a+b}{2})\right\vert ^{q}\right) ^{%
\frac{1}{q}}\right. \\
&&\left. +\left( A_{2}(\frac{2}{3},s)\left\vert f^{\prime }(b)\right\vert
^{q}+A_{3}(\frac{2}{3},s)\left\vert f^{\prime }(\frac{a+b}{2})\right\vert
^{q}\right) ^{\frac{1}{q}}\right\} ,
\end{eqnarray*}%
where 
\begin{eqnarray*}
A_{2}(\frac{2}{3},s) &=&\frac{2^{s+3}+3^{s+1}\left( s-1\right) }{%
3^{s+2}(s+1)(s+2)}, \\
A_{3}(\frac{2}{3},s) &=&\frac{2+3^{s+1}\left( 2s+1\right) }{3^{s+2}(s+1)(s+2)%
}.
\end{eqnarray*}
\end{corollary}

\begin{corollary}
Under the assumptions of Theorem \ref{2.1} with $\lambda =\frac{1}{2}$ and $%
\theta =1,$ then we have following midpoint type inequality%
\begin{eqnarray*}
&&\left\vert f\left( \frac{a+b}{2}\right) -\frac{1}{b-a}\dint%
\limits_{a}^{b}f(x)dx\right\vert \leq \frac{b-a}{8}\left( \frac{2}{(s+1)(s+2)%
}\right) ^{\frac{1}{q}} \\
&&\times \left\{ \left( \left\vert f^{\prime }(a)\right\vert ^{q}+\left(
s+1\right) \left\vert f^{\prime }(\frac{a+b}{2})\right\vert ^{q}\right) ^{%
\frac{1}{q}}\right. \\
&&\left. +\left( \left\vert f^{\prime }(b)\right\vert ^{q}+\left( s+1\right)
\left\vert f^{\prime }(\frac{a+b}{2})\right\vert ^{q}\right) ^{\frac{1}{q}%
}\right\} .
\end{eqnarray*}
\end{corollary}

\begin{corollary}
Under the assumptions of Theorem \ref{2.1} with $\lambda =\frac{1}{2}$ , and 
$\theta =0,$ then we get the following trapezoid type inequality%
\begin{eqnarray*}
&&\left\vert \frac{f\left( a\right) +f\left( b\right) }{2}-\frac{1}{b-a}%
\dint\limits_{a}^{b}f(x)dx\right\vert \leq \frac{b-a}{8}\left( \frac{2}{%
(s+1)(s+2)}\right) ^{\frac{1}{q}} \\
&&\times \left\{ \left( \left( s+1\right) \left\vert f^{\prime
}(a)\right\vert ^{q}+\left\vert f^{\prime }(\frac{a+b}{2})\right\vert
^{q}\right) ^{\frac{1}{q}}\right. \\
&&\left. +\left( \left( s+1\right) \left\vert f^{\prime }(b)\right\vert
^{q}+\left\vert f^{\prime }(\frac{a+b}{2})\right\vert ^{q}\right) ^{\frac{1}{%
q}}\right\} .
\end{eqnarray*}
\end{corollary}

Using Lemma \ref{1.1} we shall give another result for $s-$convex functions
as follows.

\begin{theorem}
\label{2.2}Let $f:I\subseteq \lbrack 0,\infty )\mathbb{\rightarrow R}$ be a
differentiable mapping on $I^{\circ }$ such that $f^{\prime }\in L[a,b]$,
where $a,b\in I^{\circ }$ with $a<b$ and $\theta ,\lambda \in \left[ 0,1%
\right] $. If $\left\vert f^{\prime }\right\vert ^{q}$ is $s-$convex on $%
[a,b]$, $q>1,$ then the following inequality holds:%
\begin{equation*}
\left\vert \left( 1-\theta \right) \left( \lambda f(a)+\left( 1-\lambda
\right) f(b)\right) +\theta f(\left( 1-\lambda \right) a+\lambda b)-\frac{1}{%
b-a}\dint\limits_{a}^{b}f(x)dx\right\vert
\end{equation*}%
\begin{eqnarray}
&\leq &\left( b-a\right) \left( \frac{\theta ^{p+1}+\left( 1-\theta \right)
^{p+1}}{p+1}\right) ^{\frac{1}{p}}  \notag \\
&&\times \left[ \lambda ^{2}\left( \frac{\left\vert f^{\prime }\left(
a\right) \right\vert ^{q}+\left\vert f^{\prime }\left( C\right) \right\vert
^{q}}{s+1}\right) ^{\frac{1}{q}}+\left( 1-\lambda \right) ^{2}\left( \frac{%
\left\vert f^{\prime }\left( b\right) \right\vert ^{q}+\left\vert f^{\prime
}\left( C\right) \right\vert ^{q}}{s+1}\right) ^{\frac{1}{q}}\right] .
\label{2-2}
\end{eqnarray}%
where $C=\left( 1-\lambda \right) a+\lambda b$ and $\frac{1}{p}+\frac{1}{q}%
=1.$
\end{theorem}

\begin{proof}
Suppose that $C=\left( 1-\lambda \right) a+\lambda b.$ From Lemma \ref{1.1}
and by H\"{o}lder's integral inequality, we have%
\begin{eqnarray*}
&&\left\vert \left( 1-\theta \right) \left( \lambda f(a)+\left( 1-\lambda
\right) f(b)\right) +\theta f(C)-\frac{1}{b-a}\dint\limits_{a}^{b}f(x)dx%
\right\vert \leq \left( b-a\right) \\
&&\left[ \lambda ^{2}\dint\limits_{0}^{1}\left\vert t-\theta \right\vert
\left\vert f^{\prime }\left( ta+\left( 1-t\right) C\right) \right\vert
dt+\left( 1-\lambda \right) ^{2}\dint\limits_{0}^{1}\left\vert t-\theta
\right\vert \left\vert f^{\prime }\left( tb+\left( 1-t\right) C\right)
\right\vert dt\right] \\
&\leq &\left( b-a\right) \left\{ \lambda ^{2}\left(
\dint\limits_{0}^{1}\left\vert t-\theta \right\vert ^{p}dt\right) ^{\frac{1}{%
p}}\left( \dint\limits_{0}^{1}\left\vert f^{\prime }\left( ta+\left(
1-t\right) C\right) \right\vert ^{q}dt\right) ^{\frac{1}{q}}\right.
\end{eqnarray*}%
\begin{equation}
\left. +\left( 1-\lambda \right) ^{2}\left( \dint\limits_{0}^{1}\left\vert
t-\theta \right\vert ^{p}dt\right) ^{\frac{1}{p}}\left(
\dint\limits_{0}^{1}\left\vert f^{\prime }\left( tb+\left( 1-t\right)
C\right) \right\vert ^{q}dt\right) ^{\frac{1}{q}}\right\} .  \label{2-2a}
\end{equation}%
Since $\left\vert f^{\prime }\right\vert ^{q}$ is $s-$convex on $[a,b],$ the
inequalities (\ref{2-11a}) and (\ref{2-11b}) holds. Hence, by simple
computation%
\begin{eqnarray}
\dint\limits_{0}^{1}\left\vert f^{\prime }\left( ta+\left( 1-t\right)
C\right) \right\vert ^{q}dt &\leq &\dint\limits_{0}^{1}t^{s}\left\vert
f^{\prime }\left( a\right) \right\vert ^{q}+\left( 1-t\right) ^{s}\left\vert
f^{\prime }\left( C\right) \right\vert ^{q}  \notag \\
&=&\frac{\left\vert f^{\prime }\left( a\right) \right\vert ^{q}+\left\vert
f^{\prime }\left( C\right) \right\vert ^{q}}{s+1}  \label{2-2b}
\end{eqnarray}%
\begin{eqnarray}
\dint\limits_{0}^{1}\left\vert f^{\prime }\left( tb+\left( 1-t\right)
C\right) \right\vert ^{q}dt &\leq &\dint\limits_{0}^{1}t^{s}\left\vert
f^{\prime }\left( b\right) \right\vert ^{q}+\left( 1-t\right) ^{s}\left\vert
f^{\prime }\left( C\right) \right\vert ^{q}  \notag \\
&=&\frac{\left\vert f^{\prime }\left( b\right) \right\vert ^{q}+\left\vert
f^{\prime }\left( C\right) \right\vert ^{q}}{s+1}  \label{2-2c}
\end{eqnarray}%
and%
\begin{equation}
\dint\limits_{0}^{1}\left\vert t-\theta \right\vert ^{p}dt=\frac{\theta
^{p+1}+\left( 1-\theta \right) ^{p+1}}{p+1}  \label{2-2d}
\end{equation}%
thus, using (\ref{2-2b})-(\ref{2-2d}) in (\ref{2-2a}), we obtain the
inequality (\ref{2-2}). This completes the proof.
\end{proof}

\begin{corollary}
Under the assumptions of Theorem \ref{2.2} with $s=1$, we have%
\begin{equation*}
\left\vert \left( 1-\theta \right) \left( \lambda f(a)+\left( 1-\lambda
\right) f(b)\right) +\theta f(\left( 1-\lambda \right) a+\lambda b)-\frac{1}{%
b-a}\dint\limits_{a}^{b}f(x)dx\right\vert
\end{equation*}%
\begin{eqnarray*}
&\leq &\left( b-a\right) \left( \frac{\theta ^{p+1}+\left( 1-\theta \right)
^{p+1}}{p+1}\right) ^{\frac{1}{p}}\left( \frac{1}{2}\right) ^{\frac{1}{q}} \\
&&\times \left[ \lambda ^{2}\left( \left\vert f^{\prime }\left( a\right)
\right\vert ^{q}+\left\vert f^{\prime }\left( C\right) \right\vert
^{q}\right) ^{\frac{1}{q}}+\left( 1-\lambda \right) ^{2}\left( \left\vert
f^{\prime }\left( b\right) \right\vert ^{q}+\left\vert f^{\prime }\left(
C\right) \right\vert ^{q}\right) ^{\frac{1}{q}}\right] ,
\end{eqnarray*}%
where $C=\left( 1-\lambda \right) a+\lambda b$ and $\frac{1}{p}+\frac{1}{q}%
=1.$
\end{corollary}

\begin{corollary}
Under the assumptions of Theorem \ref{2.2} with $\theta =1,$ then we have
the following generalized midpoint type inequality%
\begin{equation*}
\left\vert f(\left( 1-\lambda \right) a+\lambda b)-\frac{1}{b-a}%
\dint\limits_{a}^{b}f(x)dx\right\vert \leq \left( b-a\right) \left( \frac{1}{%
p+1}\right) ^{\frac{1}{p}}
\end{equation*}%
\begin{equation*}
\times \left[ \lambda ^{2}\left( \frac{\left\vert f^{\prime }\left( a\right)
\right\vert ^{q}+\left\vert f^{\prime }\left( C\right) \right\vert ^{q}}{s+1}%
\right) ^{\frac{1}{q}}+\left( 1-\lambda \right) ^{2}\left( \frac{\left\vert
f^{\prime }\left( b\right) \right\vert ^{q}+\left\vert f^{\prime }\left(
C\right) \right\vert ^{q}}{s+1}\right) ^{\frac{1}{q}}\right] ,
\end{equation*}%
where $C=\left( 1-\lambda \right) a+\lambda b$ and $\frac{1}{p}+\frac{1}{q}%
=1.$
\end{corollary}

\begin{corollary}
Under the assumptions of Theorem \ref{2.2} with $\theta =0,$ then we have
the following generalized trapezoid type inequality%
\begin{equation*}
\left\vert \lambda f(a)+\left( 1-\lambda \right) f(b)-\frac{1}{b-a}%
\dint\limits_{a}^{b}f(x)dx\right\vert \leq \left( b-a\right) \left( \frac{1}{%
p+1}\right) ^{\frac{1}{p}}
\end{equation*}%
\begin{equation*}
\times \left[ \lambda ^{2}\left( \frac{\left\vert f^{\prime }\left( a\right)
\right\vert ^{q}+\left\vert f^{\prime }\left( C\right) \right\vert ^{q}}{s+1}%
\right) ^{\frac{1}{q}}+\left( 1-\lambda \right) ^{2}\left( \frac{\left\vert
f^{\prime }\left( b\right) \right\vert ^{q}+\left\vert f^{\prime }\left(
C\right) \right\vert ^{q}}{s+1}\right) ^{\frac{1}{q}}\right] ,
\end{equation*}%
where $C=\left( 1-\lambda \right) a+\lambda b$ and $\frac{1}{p}+\frac{1}{q}%
=1.$
\end{corollary}

\begin{corollary}
Under the assumptions of Theorem \ref{2.2} with $\theta =1,$ if $\left\vert
f^{\prime }(x)\right\vert \leq M,$ $x\in \left[ a,b\right] ,$ then we have
the following Ostrowski type inequality%
\begin{equation}
\left\vert f(x)-\frac{1}{b-a}\dint\limits_{a}^{b}f(u)du\right\vert \leq 
\frac{M}{\left( p+1\right) ^{\frac{1}{p}}}\left( \frac{2}{s+1}\right) ^{%
\frac{1}{q}}\left[ \frac{\left( x-a\right) ^{2}+\left( b-x\right) ^{2}}{b-a}%
\right]  \label{2-2e}
\end{equation}%
for each $x\in \left[ a,b\right] .$
\end{corollary}

\begin{proof}
For each $x\in \left[ a,b\right] $, there exist $\lambda _{x}\in \left[ 0,1%
\right] $ such that $x=\left( 1-\lambda _{x}\right) a+\lambda _{x}b.$ Hence
we have $\lambda _{x}=\frac{x-a}{b-a}$ and $1-\lambda _{x}=\frac{b-x}{b-a}.$
Therefore for each $x\in \left[ a,b\right] ,$ from the inequality (\ref{2-2}%
) we obtain the inequality (\ref{2-2e}).
\end{proof}

\begin{remark}
We note that the inequality (\ref{2-2e}) is the same of the inequality in 
\cite[Theorem 3]{ADDC10}.
\end{remark}

\begin{corollary}
Under the assumptions of Theorem \ref{2.2} with $\lambda =\frac{1}{2}$ and $%
\theta =\frac{2}{3}$, then we have the following Simpson type inequality 
\begin{equation*}
\left\vert \frac{1}{6}\left[ f(a)+4f\left( \frac{a+b}{2}\right) +f(b)\right]
-\frac{1}{b-a}\dint\limits_{a}^{b}f(x)dx\right\vert
\end{equation*}%
\begin{equation*}
\leq \frac{b-a}{12}\left( \frac{1+2^{p+1}}{3\left( p+1\right) }\right) ^{%
\frac{1}{p}}\left\{ \left( \frac{\left\vert f^{\prime }\left( \frac{a+b}{2}%
\right) \right\vert ^{q}+\left\vert f^{\prime }\left( a\right) \right\vert
^{q}}{s+1}\right) ^{\frac{1}{q}}+\left( \frac{\left\vert f^{\prime }\left( 
\frac{a+b}{2}\right) \right\vert ^{q}+\left\vert f^{\prime }\left( b\right)
\right\vert ^{q}}{s+1}\right) ^{\frac{1}{q}}\right\} ,
\end{equation*}%
which is the same of the inequality in \cite[Theorem 8]{SSO10}.
\end{corollary}

\begin{corollary}
Under the assumptions of Theorem \ref{2.2} with $\lambda =\frac{1}{2}$ and $%
\theta =1,$ then we have the following midpoint type inequality%
\begin{eqnarray}
&&\left\vert f\left( \frac{a+b}{2}\right) -\frac{1}{b-a}\dint%
\limits_{a}^{b}f(x)dx\right\vert \leq \frac{b-a}{4}\left( \frac{1}{p+1}%
\right) ^{\frac{1}{p}}  \label{2-3} \\
&&\times \left\{ \left( \frac{\left\vert f^{\prime }\left( \frac{a+b}{2}%
\right) \right\vert ^{q}+\left\vert f^{\prime }\left( a\right) \right\vert
^{q}}{s+1}\right) ^{\frac{1}{q}}+\left( \frac{\left\vert f^{\prime }\left( 
\frac{a+b}{2}\right) \right\vert ^{q}+\left\vert f^{\prime }\left( b\right)
\right\vert ^{q}}{s+1}\right) ^{\frac{1}{q}}\right\} .  \notag
\end{eqnarray}
\end{corollary}

\begin{remark}
We note that the inequality (\ref{2-3}) is better than the inequality in 
\cite[Theorem 2.3]{ADK11}. Because, by inequality 
\begin{equation*}
2^{s-1}\left\vert f^{\prime }\left( \frac{a+b}{2}\right) \right\vert
^{q}\leq \frac{\left\vert f^{\prime }\left( a\right) \right\vert
^{q}+\left\vert f^{\prime }\left( b\right) \right\vert ^{q}}{s+1}
\end{equation*}%
we have%
\begin{eqnarray*}
\left\vert f\left( \frac{a+b}{2}\right) -\frac{1}{b-a}\dint%
\limits_{a}^{b}f(x)dx\right\vert &\leq &\left( \frac{b-a}{4}\right) \left( 
\frac{1}{p+1}\right) ^{\frac{1}{p}}\left( \frac{1}{s+1}\right) ^{\frac{2}{q}}
\\
&&\times \left[ \left( \left( 2^{1-s}+s+1\right) \left\vert f^{\prime
}\left( a\right) \right\vert ^{q}+2^{1-s}\left\vert f^{\prime }\left(
b\right) \right\vert ^{q}\right) ^{\frac{1}{q}}\right. \\
&&+\left. \left( \left( 2^{1-s}+s+1\right) \left\vert f^{\prime }\left(
b\right) \right\vert ^{q}+2^{1-s}\left\vert f^{\prime }\left( a\right)
\right\vert ^{q}\right) ^{\frac{1}{q}}\right] ,
\end{eqnarray*}%
which is the same of the inequality in \cite[Theorem 2.3]{ADK11}.
\end{remark}

\begin{corollary}
Under the assumptions of Theorem \ref{2.2} with $\lambda =\frac{1}{2}$ and $%
\theta =0,$ then we have the following trapezoid type inequality%
\begin{eqnarray}
&&\left\vert \frac{f\left( a\right) +f\left( b\right) }{2}-\frac{1}{b-a}%
\dint\limits_{a}^{b}f(x)dx\right\vert \leq \frac{b-a}{4}\left( \frac{1}{p+1}%
\right) ^{\frac{1}{p}}  \label{2-4} \\
&&\times \left\{ \left( \frac{\left\vert f^{\prime }\left( \frac{a+b}{2}%
\right) \right\vert ^{q}+\left\vert f^{\prime }\left( a\right) \right\vert
^{q}}{s+1}\right) ^{\frac{1}{q}}+\left( \frac{\left\vert f^{\prime }\left( 
\frac{a+b}{2}\right) \right\vert ^{q}+\left\vert f^{\prime }\left( b\right)
\right\vert ^{q}}{s+1}\right) ^{\frac{1}{q}}\right\} .  \notag
\end{eqnarray}%
We note that the obtained inequality (\ref{2-4}) is better than the first
inequality in \cite[Theorem 3]{KBOP07}.
\end{corollary}

\begin{theorem}
\label{2.4}Let $f:I\subseteq \left[ 0,\infty \right) \mathbb{\rightarrow R}$
be a differentiable mapping on $I^{\circ }$ such that $f^{\prime }\in L[a,b]$%
, where $a,b\in I^{\circ }$ with $a<b$ and $\alpha ,\lambda \in \left[ 0,1%
\right] $. If $\left\vert f^{\prime }\right\vert ^{q}$ is $s-$concave on $%
[a,b]$, $q>1,$ then the following inequality holds:%
\begin{eqnarray}
&&\left\vert \left( 1-\theta \right) \left( \lambda f(a)+\left( 1-\lambda
\right) f(b)\right) +\theta f(\left( 1-\lambda \right) a+\lambda b)-\frac{1}{%
b-a}\dint\limits_{a}^{b}f(x)dx\right\vert  \label{2-5} \\
&\leq &\left( b-a\right) \left( \frac{1}{2}\right) ^{\frac{1-s}{q}}\left( 
\frac{\theta ^{p+1}+\left( 1-\theta \right) ^{p+1}}{p+1}\right) ^{\frac{1}{p}%
}  \notag \\
&&\times \left\{ \lambda ^{2}\left\vert f^{\prime }\left( \frac{\left(
2-\lambda \right) a+\lambda b}{2}\right) \right\vert +\left( 1-\lambda
\right) ^{2}\left\vert f^{\prime }\left( \frac{\left( 1-\lambda \right)
a+\left( 1+\lambda \right) b}{2}\right) \right\vert \right\} ,  \notag
\end{eqnarray}%
where $C=\left( 1-\lambda \right) a+\lambda b$ and $1/p+1/q=1$.
\end{theorem}

\begin{proof}
Suppose that $C=\left( 1-\lambda \right) a+\lambda b.$ We proceed similarly
as in the proof Theorem \ref{2.2}. Since $\left\vert f^{\prime }\right\vert
^{q}$ is $s-$concave on $[a,b],$ $\lambda \in \left( 0,1\right] $ by the
inequality (\ref{1-3}), we get%
\begin{eqnarray}
&&\frac{1}{C-a}\dint\limits_{a}^{C}\left\vert f^{\prime }\left( x\right)
\right\vert ^{q}dx  \notag \\
&=&\dint\limits_{0}^{1}\left\vert f^{\prime }\left( ta+\left( 1-t\right)
C\right) \right\vert ^{q}dt\leq \frac{1}{2^{1-s}}\left\vert f^{\prime
}\left( \frac{\left( 2-\lambda \right) a+\lambda b}{2}\right) \right\vert
^{q}  \label{2-5a}
\end{eqnarray}%
the inequality (\ref{2-5a}) also holds $\lambda =0$ too. Similarly, for $%
\lambda \in \left[ 0,1\right) $ by the inequality (\ref{1-3}), we have%
\begin{eqnarray}
&&\frac{1}{b-C}\dint\limits_{C}^{b}\left\vert f^{\prime }\left( x\right)
\right\vert ^{q}dx  \notag \\
&=&\dint\limits_{0}^{1}\left\vert f^{\prime }\left( tb+\left( 1-t\right)
C\right) \right\vert ^{q}dt\leq \frac{1}{2^{1-s}}\left\vert f^{\prime
}\left( \frac{\left( 1-\lambda \right) a+\left( 1+\lambda \right) b}{2}%
\right) \right\vert ^{q}  \label{2-5b}
\end{eqnarray}%
the inequality (\ref{2-5b}) also holds $\lambda =1$ too. Thus using (\ref%
{2-2d}), (\ref{2-5a}) and (\ref{2-5b}) in (\ref{2-2a}), we obtain the
inequality (\ref{2-5}). This completes the proof.
\end{proof}

\begin{corollary}
\label{2.5}Under the assumptions of Theorem \ref{2.4} with $s=1$, we have%
\begin{eqnarray*}
&&\left\vert \left( 1-\theta \right) \left( \lambda f(a)+\left( 1-\lambda
\right) f(b)\right) +\theta f(\left( 1-\lambda \right) a+\lambda b)-\frac{1}{%
b-a}\dint\limits_{a}^{b}f(x)dx\right\vert \\
&\leq &\left( b-a\right) \left( \frac{\theta ^{p+1}+\left( 1-\theta \right)
^{p+1}}{p+1}\right) ^{\frac{1}{p}} \\
&&\times \left\{ \lambda ^{2}\left\vert f^{\prime }\left( \frac{\left(
2-\lambda \right) a+\lambda b}{2}\right) \right\vert +\left( 1-\lambda
\right) ^{2}\left\vert f^{\prime }\left( \frac{\left( 1-\lambda \right)
a+\left( 1+\lambda \right) b}{2}\right) \right\vert \right\} .
\end{eqnarray*}
\end{corollary}

\begin{remark}
In Corollary \ref{2.5}, if we take $\lambda =\frac{1}{2}$ and $\theta =0,$
then we have the following trapezoid type inequality.%
\begin{eqnarray*}
&&\left\vert \frac{f\left( a\right) +f\left( b\right) }{2}-\frac{1}{b-a}%
\dint\limits_{a}^{b}f(x)dx\right\vert \\
&\leq &\frac{b-a}{4}\left( \frac{1}{p+1}\right) ^{\frac{1}{p}}\left[
\left\vert f^{\prime }\left( \frac{3b+a}{4}\right) \right\vert +\left\vert
f^{\prime }\left( \frac{3a+b}{4}\right) \right\vert \right]
\end{eqnarray*}%
which is the same of the inequality in \cite[Theorem 2]{KBOP07}
\end{remark}

\begin{remark}
In Corollary \ref{2.5}, if we take $\lambda =\frac{1}{2}$ and $\theta =1,$
then we have the following midpoint type inequality%
\begin{eqnarray*}
&&\left\vert f\left( \frac{a+b}{2}\right) -\frac{1}{b-a}\dint%
\limits_{a}^{b}f(x)dx\right\vert \\
&\leq &\frac{b-a}{4}\left( \frac{1}{p+1}\right) ^{\frac{1}{p}}\left[
\left\vert f^{\prime }\left( \frac{3b+a}{4}\right) \right\vert +\left\vert
f^{\prime }\left( \frac{3a+b}{4}\right) \right\vert \right]
\end{eqnarray*}%
which is the same of the inequality in \cite[Theorem 2.5]{ADK11}.
\end{remark}

\begin{corollary}
Under the assumptions of Theorem \ref{2.4} with $\theta =0,$ we have the
following generalized trapezoid type inequality%
\begin{eqnarray*}
&&\left\vert \lambda f(a)+\left( 1-\lambda \right) f(b)-\frac{1}{b-a}%
\dint\limits_{a}^{b}f(x)dx\right\vert \\
&\leq &\left( b-a\right) \left( \frac{1}{2}\right) ^{\frac{1-s}{q}}\left( 
\frac{1}{p+1}\right) ^{\frac{1}{p}} \\
&&\times \left\{ \lambda ^{2}\left\vert f^{\prime }\left( \frac{\left(
2-\lambda \right) a+\lambda b}{2}\right) \right\vert +\left( 1-\lambda
\right) ^{2}\left\vert f^{\prime }\left( \frac{\left( 1-\lambda \right)
a+\left( 1+\lambda \right) b}{2}\right) \right\vert \right\} ,
\end{eqnarray*}
\end{corollary}

\begin{corollary}
Under the assumptions of Theorem \ref{2.4} with $\theta =1,$ we have the
following generalized midpoint type inequality%
\begin{eqnarray*}
&&\left\vert f(\left( 1-\lambda \right) a+\lambda b)-\frac{1}{b-a}%
\dint\limits_{a}^{b}f(x)dx\right\vert \\
&\leq &\left( b-a\right) \left( \frac{1}{2}\right) ^{\frac{1-s}{q}}\left( 
\frac{1}{p+1}\right) ^{\frac{1}{p}} \\
&&\times \left\{ \lambda ^{2}\left\vert f^{\prime }\left( \frac{\left(
2-\lambda \right) a+\lambda b}{2}\right) \right\vert +\left( 1-\lambda
\right) ^{2}\left\vert f^{\prime }\left( \frac{\left( 1-\lambda \right)
a+\left( 1+\lambda \right) b}{2}\right) \right\vert \right\} ,
\end{eqnarray*}
\end{corollary}

\begin{corollary}
Under the assumptions of Theorem \ref{2.2} with $\theta =1,$ we have the
following Ostrowski type inequality%
\begin{eqnarray}
&&\left\vert f(x)-\frac{1}{b-a}\dint\limits_{a}^{b}f(u)du\right\vert  \notag
\\
&\leq &\frac{2^{\frac{s-1}{q}}}{\left( p+1\right) ^{\frac{1}{p}}\left(
b-a\right) }\left[ \left( x-a\right) ^{2}\left\vert f^{\prime }\left( \frac{%
x+a}{2}\right) \right\vert +\left( b-x\right) ^{2}\left\vert f^{\prime
}\left( \frac{x+b}{2}\right) \right\vert \right]  \label{2-5c}
\end{eqnarray}%
for each $x\in \left[ a,b\right] .$
\end{corollary}

\begin{proof}
For each $x\in \left[ a,b\right] $, there exist $\lambda _{x}\in \left[ 0,1%
\right] $ such that $x=\left( 1-\lambda _{x}\right) a+\lambda _{x}b.$ Hence
we have $\lambda _{x}=\frac{x-a}{b-a}$ and $1-\lambda _{x}=\frac{b-x}{b-a}.$
Therefore for each $x\in \left[ a,b\right] ,$ from the inequality (\ref{2-5}%
) we obtain the inequality (\ref{2-5c}).
\end{proof}

\begin{remark}
We note that the inequality (\ref{2-5c}) is the same of the inequality in 
\cite[Theorem 5]{ADDC10}.
\end{remark}

\begin{corollary}
Under the assumptions of Theorem \ref{2.4} with $\lambda =\frac{1}{2}$ and $%
\theta =0,$ we have the following trapezoid type inequality%
\begin{eqnarray}
&&\left\vert \frac{f\left( a\right) +f\left( b\right) }{2}-\frac{1}{b-a}%
\dint\limits_{a}^{b}f(x)dx\right\vert  \label{2-6} \\
&\leq &\frac{b-a}{4}\left( \frac{1}{p+1}\right) ^{\frac{1}{p}}\left( \frac{1%
}{2}\right) ^{\frac{1-s}{q}}\left[ \left\vert f^{\prime }\left( \frac{3b+a}{4%
}\right) \right\vert +\left\vert f^{\prime }\left( \frac{3a+b}{4}\right)
\right\vert \right]  \notag
\end{eqnarray}%
which is the same of the inequality in \cite[Theorem 8 (i)]{P10}.
\end{corollary}

\begin{corollary}
Under the assumptions of Theorem \ref{2.4} with $\lambda =\frac{1}{2}$ and $%
\theta =1,$ we have the following midpoint type inequality%
\begin{eqnarray}
&&\left\vert f\left( \frac{a+b}{2}\right) -\frac{1}{b-a}\dint%
\limits_{a}^{b}f(x)dx\right\vert  \label{2-7} \\
&\leq &\frac{b-a}{4}\left( \frac{1}{p+1}\right) ^{\frac{1}{p}}\left( \frac{1%
}{2}\right) ^{\frac{1-s}{q}}\left[ \left\vert f^{\prime }\left( \frac{3b+a}{4%
}\right) \right\vert +\left\vert f^{\prime }\left( \frac{3a+b}{4}\right)
\right\vert \right]  \notag
\end{eqnarray}%
which is the same of the inequality in \cite[Theorem 8 (ii)]{P10}.
\end{corollary}

\begin{corollary}
Under the assumptions of Theorem \ref{2.4} with $\lambda =\frac{1}{2}$ and $%
\theta =\frac{2}{3},$ we have the following trapezoid type inequality%
\begin{eqnarray}
&&\left\vert \frac{1}{6}\left[ f(a)+4f\left( \frac{a+b}{2}\right) +f(b)%
\right] -\frac{1}{b-a}\dint\limits_{a}^{b}f(x)dx\right\vert  \label{2-8} \\
&\leq &\frac{b-a}{12}\left( \frac{1+2^{p+1}}{3\left( p+1\right) }\right) ^{%
\frac{1}{p}}2^{\frac{s-1}{q}}\left[ \left\vert f^{\prime }\left( \frac{3b+a}{%
4}\right) \right\vert +\left\vert f^{\prime }\left( \frac{3a+b}{4}\right)
\right\vert \right] .  \notag
\end{eqnarray}
\end{corollary}

\begin{remark}
If $\left\vert f^{\prime }\right\vert ^{q},\ q>1,$ is concave on $\left[ a,b%
\right] ,$ using the power mean inequality, we have%
\begin{eqnarray*}
\left\vert f^{\prime }\left( \lambda x+\left( 1-\lambda \right) y\right)
\right\vert ^{q} &\geq &\lambda \left\vert f^{\prime }\left( x\right)
\right\vert ^{q}+\left( 1-\lambda \right) \left\vert f^{\prime }\left(
y\right) \right\vert ^{q} \\
&\geq &\left( \lambda \left\vert f^{\prime }\left( x\right) \right\vert
+\left( 1-\lambda \right) \left\vert f^{\prime }\left( y\right) \right\vert
\right) ^{q},
\end{eqnarray*}%
$\forall x,y\in \left[ a,b\right] $ and $\lambda \in \left[ 0,1\right] .$
Hence%
\begin{equation*}
\left\vert f^{\prime }\left( \lambda x+\left( 1-\lambda \right) y\right)
\right\vert \geq \lambda \left\vert f^{\prime }\left( x\right) \right\vert
+\left( 1-\lambda \right) \left\vert f^{\prime }\left( y\right) \right\vert
\end{equation*}%
so $\left\vert f^{\prime }\right\vert $ is also concave. Then by the
inequality (\ref{1-1}), we have%
\begin{equation}
\left\vert f^{\prime }\left( \frac{3b+a}{4}\right) \right\vert +\left\vert
f^{\prime }\left( \frac{3a+b}{4}\right) \right\vert \leq 2\left\vert
f^{\prime }\left( \frac{a+b}{2}\right) \right\vert .  \label{2-9}
\end{equation}%
Thus, using the inequality (\ref{2-9}) in (\ref{2-6}), (\ref{2-7}) and (\ref%
{2-8}) we get%
\begin{eqnarray*}
\left\vert \frac{f\left( a\right) +f\left( b\right) }{2}-\frac{1}{b-a}%
\dint\limits_{a}^{b}f(x)dx\right\vert &\leq &\frac{b-a}{2}\left( \frac{1}{p+1%
}\right) ^{\frac{1}{p}}\left( \frac{1}{2}\right) ^{\frac{1-s}{q}}\left\vert
f^{\prime }\left( \frac{a+b}{2}\right) \right\vert , \\
\left\vert f\left( \frac{a+b}{2}\right) -\frac{1}{b-a}\dint%
\limits_{a}^{b}f(x)dx\right\vert &\leq &\frac{b-a}{2}\left( \frac{1}{p+1}%
\right) ^{\frac{1}{p}}\left( \frac{1}{2}\right) ^{\frac{1-s}{q}}\left\vert
f^{\prime }\left( \frac{a+b}{2}\right) \right\vert ,
\end{eqnarray*}%
and%
\begin{eqnarray*}
&&\left\vert \frac{1}{6}\left[ f(a)+4f\left( \frac{a+b}{2}\right) +f(b)%
\right] -\frac{1}{b-a}\dint\limits_{a}^{b}f(x)dx\right\vert \\
&\leq &\frac{b-a}{12}\left( \frac{1+2^{p+1}}{3\left( p+1\right) }\right) ^{%
\frac{1}{p}}\left( \frac{1}{2}\right) ^{\frac{1-s}{q}}\left\vert f^{\prime
}\left( \frac{a+b}{2}\right) \right\vert ,
\end{eqnarray*}%
respectively, where $s\in \left( 0,1\right] $ and $1/p+1/q=1$.
\end{remark}

\section{Some applications for special means}

in \cite{HM94}, the following example is given.

\bigskip Let $s\in \left( 0,1\right) $ and $a,b,c\in 
\mathbb{R}
.$We define function $f:\left[ 0,\infty \right) \rightarrow 
\mathbb{R}
$ as%
\begin{equation*}
f(t)=\left\{ 
\begin{array}{cc}
a, & t=0 \\ 
bt^{s}+c, & t>0%
\end{array}%
\right. .
\end{equation*}%
If $b\geq 0$ and $0\leq c\leq a,$ then $\ f\in K_{s}^{2}.$ Hence, for $%
a=c=0,\ b=s+1$, $s\in \left( 0,\frac{1}{q}\right) $, $q\geq 1,$ we have $f:%
\left[ 0,\infty \right) \rightarrow \left[ 0,\infty \right) $, $f(t)=t^{s+1}$%
, $\ \left\vert f^{\prime }\right\vert ^{q}\in K_{s}^{2}.$

Let us recall the following special means of arbitrary real numbers $a,b$
with $a\neq b$ and $\alpha \in \left[ 0,1\right] :$

\begin{enumerate}
\item The weighted arithmetic mean%
\begin{equation*}
A_{\alpha }\left( a,b\right) :=\alpha a+(1-\alpha )b,~a,b\in 
\mathbb{R}
.
\end{equation*}

\item The unweighted arithmetic mean%
\begin{equation*}
A\left( a,b\right) :=\frac{a+b}{2},~a,b\in 
\mathbb{R}
.
\end{equation*}

\item The Logarithmic mean%
\begin{equation*}
L\left( a,b\right) :=\frac{b-a}{\ln b-\ln a},\ \ a\neq b,\ \ a,b>0.
\end{equation*}

\item Then $p-$Logarithmic mean%
\begin{equation*}
L_{p}\left( a,b\right) :=\ \left( \frac{b^{p+1}-a^{p+1}}{(p+1)(b-a)}\right)
^{\frac{1}{p}}\ ,\ p\in 
\mathbb{R}
\backslash \left\{ -1,0\right\} ,\ a,b>0.
\end{equation*}
\end{enumerate}

Now, using the results of Section 2, some new inequalities are derived for
the above means.

\begin{proposition}
Let $a,b\in 
\mathbb{R}
$ with $0<a<b,\ q\geq 1$,\ $s\in \left( 0,\frac{1}{q}\right) $ and $\lambda
,\theta \in \left[ 0,1\right] $ we have the following inequality:%
\begin{eqnarray*}
&&\left\vert \left( 1-\theta \right) A_{\lambda }\left(
a^{s+1},b^{s+1}\right) +\theta A_{\lambda }^{s+1}\left( a,b\right)
-L_{s+1}^{s+1}\left( a,b\right) \right\vert \\
&\leq &\left( b-a\right) A_{1}^{1-\frac{1}{q}}(\theta )\left( s+1\right)
\left\{ \lambda ^{2}\left[ a^{sq}A_{2}(\theta ,s)+A_{\lambda
}^{sq}(b,a)A_{3}(\theta ,s)\right] ^{\frac{1}{q}}\right. \\
&&\left. +\left( 1-\lambda \right) ^{2}\left[ b^{sq}A_{2}(\theta
,s)+A_{\lambda }^{sq}(b,a)A_{3}(\theta ,s)\right] ^{\frac{1}{q}}\right\}
\end{eqnarray*}%
where $A_{1}(\theta ),\ A_{2}(\theta ,s),\ A_{3}(\theta ,s)$ are defined as
in Theorem \ref{2.1}.
\end{proposition}

\begin{proof}
The assertion follows from Theorem \ref{2.1} for the function $f(t)=t^{s+1}$,%
$\ t\in \left[ 0,\infty \right) $, $s\in \left( 0,\frac{1}{q}\right) .$
\end{proof}

\begin{proposition}
Let $a,b\in 
\mathbb{R}
$ with $0<a<b,\ p,q>1,\ \frac{1}{p}+\frac{1}{q}=1$,\ $s\in \left( 0,\frac{1}{%
q}\right) $ and $\lambda ,\theta \in \left[ 0,1\right] $ we have the
following inequality:%
\begin{eqnarray*}
&&\left\vert \left( 1-\theta \right) A_{\lambda }\left(
a^{s+1},b^{s+1}\right) +\theta A_{\lambda }^{s+1}\left( a,b\right)
-L_{s+1}^{s+1}\left( a,b\right) \right\vert \\
&\leq &\left( b-a\right) \left( \frac{\theta ^{p+1}+\left( 1-\theta \right)
^{p+1}}{p+1}\right) ^{\frac{1}{p}}\left( s+1\right) ^{1-\frac{1}{q}} \\
&&\times \left[ \lambda ^{2}\left( a^{sq}+A_{\lambda }^{sq}(b,a)\right) ^{%
\frac{1}{q}}+\left( 1-\lambda \right) ^{2}\left( b^{sq}+A_{\lambda
}^{sq}(b,a)\right) ^{\frac{1}{q}}\right] .
\end{eqnarray*}%
where $1/p+1/q=1$.
\end{proposition}

\begin{proof}
The assertion follows from Theorem \ref{2.2} for the function $f(t)=t^{s+1}$,%
$\ t\in \left[ 0,\infty \right) $, $s\in \left( 0,\frac{1}{q}\right) .$
\end{proof}

\end{document}